\documentclass[11pt,reqno]{amsart}

\usepackage{a4wide}
\usepackage{amsmath,amssymb}
\usepackage{bbm}
\usepackage{xcolor}
\usepackage{latexsym}
\usepackage{tikz-cd}

\allowdisplaybreaks

\usepackage{scalerel,stackengine}
\stackMath
\newcommand\reallywidehat[1]{%
\savestack{\tmpbox}{\stretchto{%
  \scaleto{%
    \scalerel*[\widthof{\ensuremath{#1}}]{\kern-.6pt\bigwedge\kern-.6pt}%
    {\rule[-\textheight/2]{1ex}{\textheight}}
  }{\textheight}%
}{0.5ex}}%
\stackon[1pt]{#1}{\tmpbox}%
}

\newcommand\reallywidecheck[1]{%
\savestack{\tmpbox}{\stretchto{%
  \scaleto{%
    \scalerel*[\widthof{\ensuremath{#1}}]{\kern-.6pt\bigwedge\kern-.6pt}%
    {\rule[-\textheight/2]{1ex}{\textheight}}
  }{\textheight}%
}{0.5ex}}%
\stackon[1pt]{#1}{\scalebox{-1}{\tmpbox}}%
}

\numberwithin{equation}{section}

\newcommand{\Z}{{\mathbb Z}}

\newcommand{\R}{{\mathbb R}}
\newcommand{\N}{{\mathbb N}}
\newcommand{\CC}{{\mathbb C}}

\newcommand{\supp}{{\mbox{supp}}}

\newcommand{\mc}{\mathcal}

\newcommand{\dd}{\mbox{d}}

\newcommand{\eps}{\varepsilon}
\newcommand{\cM}{{\mathcal M}}
\newcommand{\cL}{{\mathcal L}}

\newcommand{\cF}{{\mathcal F}}

\newcommand{\SAP}{\mathcal{S}\hspace*{-2pt}\mathcal{AP}}

\newcommand{\NAP}{\mathcal{N}\hspace*{-1pt}\mathcal{AP}}

\newcommand{\lm}{\ensuremath{\lambda\!\!\!\lambda}}

\newcommand{\Cu}{C_{\mathsf{u}}}
\newcommand{\Cc}{C_{\mathsf{c}}}
\newcommand{\Cz}{C^{}_{0}}

\newcommand{\limn}{\lim_{n\to\infty}}
\newcommand{\lims}{\limsup_{n\to\infty}}

\newcommand{\exend}{\hfill$\Diamond$}

\theoremstyle{plain}
\newtheorem{theorem}{Theorem}[section]
\newtheorem{prop}[theorem]{Proposition}
\newtheorem{lemma}[theorem]{Lemma}
\newtheorem{coro}[theorem]{Corollary}

\theoremstyle{definition}
\newtheorem{definition}[theorem]{Definition}
\newtheorem{remark}[theorem]{Remark}
\newtheorem{example}[theorem]{Example}

\begin{document}
\title[Norm almost periodic measures]{On norm almost periodic measures}

\author{Timo Spindeler}
\address{Department of Mathematical and Statistical Sciences, \newline
\hspace*{\parindent}632 CAB, University of Alberta, Edmonton, AB, T6G 2G1, Canada}
\email{spindele@ualberta.ca}

\author{Nicolae Strungaru}
\address{Department of Mathematical Sciences, MacEwan University \newline
\hspace*{\parindent} 10700 -- 104 Avenue, Edmonton, AB, T5J 4S2, Canada\\
and \\
Institute of Mathematics ``Simon Stoilow''\newline
\hspace*{\parindent}Bucharest, Romania}
\email{strungarun@macewan.ca}
\urladdr{http://academic.macewan.ca/strungarun/}

\begin{abstract}
In this paper, we study norm almost periodic measures on locally compact Abelian groups. First, we show that the norm almost periodicity of $\mu$ is equivalent to the equi-Bohr almost periodicity of $\mu*g$ for all $g$ in a fixed family of functions. Then, we show that, for absolutely continuous measures, norm almost periodicity is equivalent to the Stepanov almost periodicity of the Radon--Nikodym density.
\end{abstract}



\maketitle

\section{Introduction}

The discovery of quasicrystals in Nature \cite{She} emphasised the need for a better understanding of physical diffraction, especially for systems with pure point spectrum. Over the last two decades, tremendous amount of work has been done in this direction, and the connection between pure point diffraction and almost periodicity has become clear (see for example \cite{bm,BL,Gou,LMS,LR,LSS,LS,LS2,MoSt,SOL,SOL1,NS11,ST} to name a few).

Given a translation bounded measure $\omega$, its diffraction is defined as the Fourier transform $\widehat{\gamma}$ of the autocorrelation measure $\gamma$ \cite{Hof1} (see the monographs \cite{TAO,TAO2,KLS} for general background of mathematical diffraction theory). We say that $\omega$ is pure point diffractive if the diffraction measure $\widehat{\gamma}$ is a pure point measure, and this is equivalent to the strong almost periodicity of the autocorrelation measure $\gamma$ \cite{bm,ARMA,MoSt}. For measures with Meyer set support, this is also equivalent to the (simpler to check) norm almost periodicity of $\gamma$ \cite{bm,NS11}. This makes strong and norm almost periodicity interesting for us.

While strong almost periodicity seems to be the natural concept to study due to the direct connection with pure point diffraction, norm almost periodicity appeared in a natural way in the study of measures coming from cut and project schemes \cite{NS11}, and diffraction of measures with Meyer set support \cite{NS12}. Because of this, a better understanding of norm almost periodicity becomes important. It is known that norm almost periodicity is a stronger concept than strong almost periodicity \cite{bm}, and that for measures with Meyer set support the two concepts are equivalent \cite{bm,NS11}. This suggests that there is a deeper connection between these two concepts, a connection which has not been investigated, yet. It is our goal in this paper to look closer at the relation between these two forms of almost periodicity.

Recall that a translation bounded measure $\mu$ is called strongly almost periodic if, for each compactly supported continuous function $f$, the convolution $\mu*f$ is a Bohr almost periodic function. In Theorem~\ref{thm:char_nap}, we prove that a translation bounded measure $\mu$ is norm almost periodic if and only if the set $\{ \mu *f\ |\  f \mbox{ continuous},\, \| f \|_\infty \leqslant 1,\, \supp(f) \subseteq U \}$, where $U$ is a fixed but arbitrary precompact open set, is equi-Bohr almost periodic (meaning that, for each $\eps >0$, the set of common $\eps$-almost periods of the entire family is relatively dense).
To achieve this characterisation, we provide in Corollary~\ref{C1} and Proposition~\ref{P1} new formulas for $\| \mu \|_U$. We want to emphasise here that, while in the literature this norm is typically defined using a compact set $K$ with non-empty interior, our choice of working with precompact open sets leads to simpler and more useful formulas (see Corollary~\ref{C1} and Proposition~\ref{P1}), and therefore it is, in our opinion, more useful. Moreover, any two precompact sets $X,Y$ with non-empty interior define equivalent norms $\| \cdot \|_X$ and $\| \cdot \|_Y$, respectively, and therefore the choice of a compact set $K$ with non-empty interior or a precompact open set $U$ is irrelevant for the concept of norm almost periodicity.

The second goal of the paper is to study the norm almost periodicity of absolutely continuous measures. We show that, given an absolutely continuous measure $\mu$ with density function $f \in L^1_{\text{loc}}(G)$, the measure $\mu$ is norm almost periodic if and only if $f$ is a Stepanov almost periodic function. We also prove that if the density function is uniformly continuous and bounded, then norm almost periodicity of $\mu$ is also equivalent to the Bohr almost periodicity of $f$ and to the strong almost periodicity of $\mu$.

\medskip

The paper is structured as follows. In Section~\ref{on norm}, we provide in Corollary~\ref{C1}, Proposition~\ref{P1} and Corollary~\ref{C4} various estimates for the norm of a measure. We also prove that the spaces of translation bounded pure point measures, translation bounded absolutely continuous measures and translation bounded singularly continuous measures, respectively, are Banach spaces with respect to this norm. We complete this section by showing that these spaces are not closed with respect to the product topology.

In Section~\ref{SAP vs NAP} we study the connection between norm and strong almost periodicity we mentioned above. We prove one of the main results of the paper in the following Theorem.

\medskip

\noindent \textbf{Theorem~\ref{thm:char_nap}.}
\textit{Let $\mu \in \cM^\infty(G)$, let $U\subseteq G$ be an open precompact set, and let $\mathcal F \subseteq \cF_U$ be dense in $(\cF_U, \| \cdot \|_\infty)$. Then, $\mu$ is norm almost periodic if and only if $\mathcal{G}_{\cF}:=\{ \mu *g\ |\ g \in \cF \}$ is equi-Bohr almost periodic.}

\textit{In particular, $\mu$ is norm almost periodic if and only if the family $\mathcal{G}:=\{ \mu *g\ |\ g \in \cF_U \}$ is equi-Bohr almost periodic.}

\smallskip Here and below, for a precompact open set $U$, the set $\cF_U$ is defined as
\[
\cF_U:= \{g \in \Cc(G)\ |\ |g| \leqslant 1_U \}= \{ g \in \Cc(G)\ |\ \supp(g) \subseteq U,\, \| g\|_\infty \leqslant 1 \} \,.
\]

After that we provide examples of measures $\mu \in \SAP(G)$ for which $ \mu_{\text{pp}}, \mu_{\text{ac}}$ and/or $\mu_{\text{sc}}$ are not strongly almost periodic, see Section~\ref{sap leb}. This is interesting, since norm almost periodicity carries throughout the Lebesgue decomposition by Corollary~\ref{coro:2}.

\smallskip

In Section~\ref{nap sp}, we take a closer look at norm almost periodic measures of spectral purity. Of special interest to us are norm almost periodic absolutely continuous measures. Here we prove the second main result in the paper.

\medskip

\noindent \textbf{Theorem~\ref{T1}.}
\textit{An absolutely continuous translation bounded measure $\mu=f\,\theta_G$ is norm almost periodic if and only if its density function $f \in L^1_{\text{loc}}(G)$ is $L^1$-Stepanov almost periodic.}

\textit{The mapping $f \mapsto f\, \theta_G$ is a norm preserving isomorphism between the Banach spaces $(\mathcal{S}, \| \cdot \|_U)$ and $(\NAP_{\text{ac}}(G), \| \cdot \|_U)$, where
\[
\mathcal{S}:= \{ f \in L^1_{\text{loc}}(G) \ |\ f \mbox{ is } L^1 \mbox{-Stepanov almost periodic} \} \,.
\]}

We complete the paper by looking at some consequences of these results for the diffraction of measures with Meyer set support.

\section{Preliminaries}
Throughout the paper, $G$ denotes a second countable, locally compact (Hausdorff) Abelian group. The metric on $G$ can be chosen such that it is translation invariant and all the balls are precompact \cite{STRUB}, and we assume that this holds. The associated Haar measure is denoted by $|\cdot|$ or $\theta_G$.

We use the familiar symbols $C_{\text{c}}(G)$ and $C_{\text{u}}(G)$ for the spaces of compactly supported continuous and bounded uniformly continuous functions, respectively, which map from $G$ to $\CC$. For any function $g$ on $G$, the functions $T_tg$ and $g^{\dagger}$ are defined by
\begin{displaymath}
(T_tg)(x):=g(x-t)\quad \text{ and } \quad g^{\dagger}(x):=g(-x).
\end{displaymath}
A \textbf{measure} $\mu$ on $G$ is a linear functional on $C_{\text{c}}(G)$ such that, for every compact subset $K\subseteq G$, there is a constant $a_K>0$ with
\begin{displaymath}
|\mu(g)| \leqslant a_{K}\, \|g\|_{\infty}
\end{displaymath}
for all $g\in C_{\text{c}}(G)$ with $\supp(g) \subseteq K$. Here, $\|g\|_{\infty}$ denotes the supremum norm of $g$. By the Riesz Representation theorem \cite{Reiter,ReiterSte,RUD2}, this definition is equivalent to the classical measure theory concept of regular Radon measure.

For a measure $\mu$ on $G$, we define  $T_t\mu$ and $\mu^{\dagger}$ by
\begin{displaymath}
(T_t\mu)(g):= \mu(T_{-t}g)\quad  \text{ and } \quad
\mu^{\dagger}(g):= \mu(g^{\dagger}).
\end{displaymath}
\smallskip

Given a measure $\mu$, there exists a positive measure $\left| \mu \right|$ such that, for all $f \in \Cc(G)$ with $f \geqslant 0$, we have \cite{Ped}
\[
\left| \mu \right| (f)= \sup \{ \left| \mu (g) \right| \ |\ g \in \Cc(G),\,  |g| \leqslant f \} \,.
\]
The measure $\left| \mu \right|$ is called the \textbf{total variation of} $\mu$.

\smallskip

Recall that a measure $\mu$ on $G$ is called \textbf{translation bounded} if $\sup_{t\in G}|\mu|(t+K) < \infty$ holds for every compact subset $K\subseteq G$. The space of all translation bounded measures on $G$ is denoted by $\mc{M}^{\infty}(G)$. We will denote by $\cM^\infty_{\text{pp}}(G), \cM^\infty_{\text{ac}}(G)$ and $\cM^\infty_{\text{sc}}(G)$ the spaces of translation bounded pure point, translation bounded absolutely continuous and translation bounded singular continuous measures, respectively.

\medskip

Now, as mentioned in the Introduction, there are different notions of almost periodicity.

\begin{definition} \label{def:1}
A function $f \in C_{\text{u}}(G)$ is called \textbf{strongly} \textbf{almost periodic} if the closure of $\{T_tf\ |\ t\in G\}$ is compact in the Banach space $(\Cu(G), \| \cdot \|_\infty)$. The spaces of strongly almost periodic functions on $G$ is denoted by $\operatorname{SAP}(G)$.
\end{definition}

\begin{remark}
Note that a function $f\in C_{\text{u}}(G)$ is strongly almost periodic if and only if it is Bohr almost periodic, i.e. for each $\eps>0$, the set
\begin{displaymath}
\{t\in G\ |\ \|T_tf-f\|_{\infty} < \eps\}
\end{displaymath}
is relatively dense \cite[Prop. 4.3.2]{TAO2}.  \exend
\end{remark}

Definition~\ref{def:1} carries over to measures.

\begin{definition}
A measure $\mu$ is called \textbf{strongly almost periodic} if, for all $f \in C_{\text{c}}(G)$, the function $f*\mu$ is a strongly almost periodic function. We will denote by $\SAP(G)$ the space of all strongly almost periodic measures.
\end{definition}

Later, we will compare this notion of almost periodicity with the following stronger version.

\begin{definition}
Let $K\subseteq G$ be a compact subset with non-empty interior. A measure $\mu\in\mc{M}^{\infty}(G)$ is called \textbf{norm almost periodic} if, for all $\varepsilon > 0$, the set
\[
P_{\varepsilon}^K(\mu) := \{t \in G \ |\ \|\mu -T_t\mu\|_K < \varepsilon\}
\]
is relatively dense in $G$. The space of norm almost periodic measures will be denote by $\NAP(G)$. Here, for a translation bounded measure $\nu \in \cM^\infty(G)$, its $K$-norm (see \cite{bm,NS11} for more details and properties of this) is defined as
\[
\| \nu \|_K:= \sup_{x \in G} \left| \nu \right|(x+K) \,.
\]
\end{definition}

Last but not least, we need to define the convolution of two measures.

\begin{definition}
Let $\mu$ and $\nu$ be two measures on $G$. We say that $\mu$ and $\nu$ are \textbf{convolvable} whenever their \textbf{convolution}
\[
(\mu*\nu)(f) = \int_{G} \int_{G} f(x+y)\ \dd\mu(x)\, \dd\nu(y)
\]
exists for all $f\in C_{\text{c}}(G)$.
\end{definition}

\begin{definition} A sequence $(A_n)_{n\in\N}$ of precompact open subsets of $G$ is called a \textbf{van Hove sequence} if, for each compact set $K \subseteq G$, we have
\[
\lim_{n\to\infty} \frac{|\partial^{K} A_{n}|}{|A_{n}|}  =  0 \, ,
\]
where the \textbf{$K$-boundary $\partial^{K} A$} of an open set $A$ is defined as
\[
\partial^{K} A := \bigl( \overline{A+K}\setminus A\bigr) \cup
\bigl((\left(G \backslash A\right) - K)\cap \overline{A}\, \bigr) \,.
\]
\end{definition}

Note that every $\sigma$-compact locally compact Abelian group $G$ admits van Hove sequences \cite{Martin2}.

At the end of this section, let us review the standard notions of convergence for measures which we will use below.

\begin{definition}
Let $(\mu_n)_{n\in\N}$ be a sequence of measures on $G$, and let $\mu\in\cM(G)$. Then, the sequence $(\mu_n)_{n\in\N}$ converges to $\mu$
\begin{enumerate}
\item[$\bullet$] in the \textbf{vague topology} if $\lim_{n\to\infty} \mu_n(f)=\mu(f)$ for all $f\in \Cc(G)$;
\item[$\bullet$] in the \textbf{norm topology} if $\lim_{n\to\infty} \|\mu_n-\mu\|_K=0$ for some (fixed) non-empty and compact set $K\subseteq G$ which is the closure of its interior;
\item[$\bullet$] in the \textbf{product topology} if $\lim_{n\to\infty} \|(\mu_n-\mu)*g\|_{\infty}=0$ for all $g\in \Cc(G)$.
\end{enumerate}
These types of convergence are denoted by $\mu_n\to\mu$, $\mu_n\Rightarrow\mu$, $\mu_n\xrightarrow{\pi}\mu$.
\end{definition}

\section{On the norm of measures}\label{on norm}

In this section, we give various estimates on the norm $\| \mu \|_U$ of a measure. Let us start with the following lemma.

\begin{lemma} \label{lem:1}
Let $U$ be an open precompact set, and let $\mu$ a measure on $G$. Then,
\begin{displaymath}
\left| \mu \right| (U)= \sup \{ | \mu (g)|\ |\ g \in \Cc(G) ,\, |g| \leqslant 1_U \}.
\end{displaymath}
\end{lemma}
\begin{proof}
$\geqslant $: First, for any such $g$, we have
\begin{displaymath}
\left| \mu (g) \right| \leqslant \left| \mu \right| (|g|) \leqslant \left| \mu \right| (1_U) = \left| \mu \right| (U) \,.
\end{displaymath}

\medskip

\noindent $\leqslant$: Let $\varepsilon >0$ by arbitrary.
By the inner regularity of $|\mu|$, there exists a compact set $K \subseteq U$ such that
\begin{displaymath}
\left| \mu \right| (U) \leqslant \left| \mu \right| (K) +\frac{\varepsilon}{2} \,.
\end{displaymath}
Next, we can find some $f \in \Cc(G)$ such that $1_K \leqslant f \leqslant 1_U$, and hence
\begin{displaymath}
\left| \mu \right| (U) \leqslant \left| \mu \right| (f) +\frac{\varepsilon}{2} \,.
\end{displaymath}
Now, since $f \geqslant 0$, we have
\begin{displaymath}
\left| \mu \right| (f)= \sup\{ | \mu (g)|\ |\ g \in \Cc(G),\, |g| \leqslant f \} \,.
\end{displaymath}
Therefore, there exists a function $g \in \Cc(G)$ such that $|g| \leqslant f$ and
\begin{displaymath}
\left| \mu \right| (f)\leqslant  \left| \mu (g) \right| +\frac{\varepsilon}{2} \,.
\end{displaymath}
Thus, one has
\[
 \left| \mu (g) \right| \geqslant \left| \mu \right| (f) -\frac{\eps}{2} \geqslant \left| \mu \right| (U) -\eps \,,
\]
and $|g| \leqslant f \leqslant 1_U$. Since $\eps>0$ was arbitrary, this proves the claim.
\end{proof}

As we will often deal with functions of this type, we will use the following notation:
\[
\cF_U:= \{g \in \Cc(G)\ |\ |g| \leqslant 1_U \}= \{ g \in \Cc(G)\ |\ \supp(g) \subseteq U,\, \| g\|_\infty \leqslant 1 \} \,.
\]
As a consequence we get the following simple result, which will be important in our study of norm almost periodicity.

\begin{coro}\label{C1}
Let $U \subseteq G$ be an open and precompact subset. Then, for all $ \mu \in \cM^\infty(G)$, we have
\begin{displaymath}
\| \mu \|_U = \sup_{t \in G} \sup_{ g \in \cF_U } \left| \mu(T_t g) \right| = \sup_{ (t,g) \in G \times \cF_U} \left| \mu(T_t g) \right| = \sup_{ g \in \cF_U } \sup_{t \in G} \left| \mu(T_t g) \right|.
\end{displaymath}
\end{coro}
\begin{proof}
The first equality follows from Lemma~\ref{lem:1}. The second and third equality follow from standard properties of the supremum.
\end{proof}

We next show that each measure $\mu$ induces an operator $T_\mu$ on the space of continuous functions supported inside $-U$, and that $\| \mu \|_U$ is just the operator norm $\| T_{\mu}\|$. This enables us to give alternate formulas for $\| \mu \|_U$.
For simplicity, we write
\begin{displaymath}
C(G:U):= \{ f \in \Cc(G)\ |\ \supp(f) \subseteq U \} \,.
\end{displaymath}

\begin{prop}\label{P1}
Let $\mu \in \cM^\infty(G)$, and let $U$ be an open precompact set. Define the operator $T_{\mu}$ by
\begin{displaymath}
T_\mu: (C(G:-U), \| \cdot \|_\infty) \to (\Cu(G), \| \cdot \|_\infty) , \quad \
f \mapsto\mu*f \,.
\end{displaymath}
Then, one has
\begin{displaymath}
\|T_\mu\| = \| \mu \|_U \,.
\end{displaymath}
In particular, this gives
\begin{align*}
\| \mu \|_U
    &= \sup \{ \| \mu*f \|_\infty\ |\ f \in C(G:-U),\, \| f\|_\infty =1 \} \\
    &= \sup \{ \| \mu*f \|_\infty\ |\ f \in C(G:-U),\, \| f\|_\infty
         \leqslant 1 \}\\
    &= \sup \{ \frac{\| \mu*f \|_\infty}{ \|f \|_\infty}\ |\ f \in C(G:-U),\, f
       \not\equiv 0 \} \\
    &= \inf \{ C\ |\ \| \mu*f \|_\infty \leqslant C\, \|f\|_\infty \,
        \text{ for all } f\in C(G:-U) \} \,.
\end{align*}
\end{prop}
\begin{proof}

First note that $g \mapsto g^\dagger$ defines an isometric isomorphism between $C(G:-U)$ and $C(G:U)$. It follows immediately from Corollary~\ref{C1} that
\begin{displaymath}
\| \mu \|_U
    =  \sup_{ g \in \cF_U }\sup_{t \in G} \big| \mu(T_t g) \big|  =
        \sup_{ g \in \cF_{-U} } \sup_{t \in G}  \big| \mu(T_t g^\dagger)
         \big|
    =\sup_{ g \in \cF_{-U} } \sup_{t \in G}  \big|(\mu*g)(t) \big|=
        \sup_{ g \in \cF_{-U} }   \|\mu*g \|_\infty \,.
\end{displaymath}
This yields
\begin{equation}\label{eq2}
\| \mu \|_U= \sup_{ g \in \cF_{-U} }   \|\mu*g \|_\infty \,.
\end{equation}
Now, since $\mu \in \cM^\infty(G)$, we have $\mu*g \in \Cu(G)$ for all $g \in \Cc(G)$ \cite{ARMA,MoSt}. Therefore, $T_\mu$ is well defined, and it is easy to see that $T_\mu$ is linear.

Next, we have
\begin{displaymath}
\cF_{-U} = \{ g \in \Cc(G)\ |\ \supp(g) \subseteq -U,\, \| g\|_\infty \leqslant 1 \} = \{ g \in C(G:-U)\ |\ \| g\|_\infty \leqslant 1 \}\,.
\end{displaymath}
Hence, $\cF_{-U}$ is the unit ball in the normed space $( C(G:-U), \| \cdot \|_\infty)$. Therefore, by Eq.~\eqref{eq2}, we get
\begin{displaymath}
\| \mu \|_U=\sup \{ \| T_{\mu}(f) \|_\infty\ |\ f \in C(G:-U),\, \| f\|_\infty \leqslant 1 \}= \| T_\mu\| \,.
\end{displaymath}

Finally, the last claim follows from standard equivalent definitions of the operator norm on normed spaces.
\end{proof}

As an immediate consequence, we obtain the next result.

\begin{coro}\label{C5} Let $U$ be an open precompact set, and let $\mathcal F \subseteq \mathcal F_{-U}$ be dense in $(\mathcal F_{-U}, \| \cdot \|_\infty)$. Then, one has
\begin{displaymath}
\| \mu \|_U= \sup \{ \| \mu*f \|_\infty\ |\ f \in \mathcal F \}
\end{displaymath}
\end{coro}
\begin{proof}
With the notation of Proposition~\ref{P1}, since $\mathcal F$ is dense in $(\mathcal F_{-U}, \| \cdot \|_\infty)$ and $(\mathcal F_{-U}, \| \cdot \|_\infty)$ is the unit ball in $(C(G:-U), \| \cdot \|_\infty)$, we get:
\begin{displaymath}
 \| \mu \|_U =\|T_\mu\| = \sup_{ f \in \mathcal F} \| T_\mu (f) \| =  \sup \{ \| \mu*f \|_\infty\ |\ f \in \mathcal F \} \,.      \qedhere
\end{displaymath}
\end{proof}

We next provide a similar estimate for the norm for compact sets, via approximations from above. Let us start with a preliminary lemma.

\begin{lemma}
Let $\mu$ be a positive measure, and let $B$ be a precompact Borel set. Then, we have
\begin{displaymath}
\mu(\overline B) = \inf\{ \mu(f)\ |\ f\in\Cc(G),\, f\geqslant 1_B\}.
\end{displaymath}
\end{lemma}
\begin{proof}
On the one hand, we have
\begin{displaymath}
\mu(\overline B) = \mu(1_{\overline B}) \leqslant \mu(f)
\end{displaymath}
for all $f\in \Cc(G)$ with $f\geqslant 1_B$, since $f\in \Cc(G)$ and $f\geqslant 1_B$ imply $f\geqslant 1_{\overline B}$. Hence, we obtain $ \mu(\overline B) \leqslant \inf\{ \mu(f)\ |\ f\in\Cc(G),\, f\geqslant 1_B\}$.

On the other hand, we have

\begin{align*}
\mu(\overline B)
    &= \inf\{\mu(U)\ |\ \overline B\subseteq U,\, U \text{ open}\} =\inf\{ \mu(1_U)\ |\ \overline B\subseteq U,\, U \text{ open}\} \\
    &\geqslant \inf\{ \mu(f)\ |\ \overline B\subseteq U,\, U \text{ open},
       \, f\in \Cc(G),\, 1_U\geqslant f\geqslant 1_B\} \\
    &\geqslant \inf\{ \mu(f)\ |\ \, f\in \Cc(G),\,  f\geqslant 1_B\}\,.
\end{align*}
Therefore, the claim follows.
\end{proof}

Consequently, we get the following estimates.

\begin{coro} \label{coro:a}
Let $\mu$ be a positive measure, and let $K\subseteq G$ be a compact set. Then, we have
\begin{displaymath}
\mu(K) = \inf\{ \mu(f)\ |\ f\in\Cc(G),\, f\geqslant 1_K\}.
\end{displaymath}
\end{coro}

The next corollary is an immediate consequence.

\begin{coro}\label{C4}
Let $\mu$ be a measure on $G$, and let $K\subseteq G$ be a compact set. Then, we have
\begin{displaymath}
\|\mu\|_K = \sup_{t\in G}\, \inf_{\substack{f\in \Cc(G),\\ f\geqslant 1_K}} \, \left| \mu \right| (T_tf) \,.
\end{displaymath}
In particular, if $\mu$ is positive, then we have
\begin{displaymath}
\|\mu\|_K = \sup_{t\in G}\, \inf_{\substack{f\in \Cc(G),\\ f\geqslant 1_K}} \,  \mu  (T_tf) \,.
\end{displaymath}
\end{coro}
\begin{proof}
This follows from Corollary~\ref{coro:a} because
\begin{displaymath}
\|\mu\|_K = \sup_{t\in G}\left|  \mu \right| (t+K) = \sup_{t\in G} \, \inf_{\substack{f\in
\Cc(G) ,\\ f\geqslant 1_{t+K}}} \, \left|\mu\right|(f) =  \sup_{t\in G} \, \inf_{\substack{f\in \Cc(G) ,\\ f\geqslant 1_K}} \, \left|\mu\right|(T_tf) \,.
\qedhere
\end{displaymath}
\end{proof}

\begin{remark} When working with precompact open sets, the formula of Corollary~\ref{C1} involves two suprema, which can be interchanged. In contrast, the supremum and infimum in Corollary~\ref{C4} cannot be interchanged. Because of this, it is much easier to work with open precompact sets when estimating $\| \mu \|_U$ than with compact sets, and this is why we make this choice below.    \exend
\end{remark}

Let us emphasise that our choice of open precompact sets does not matter when working with the norm topology. The following result is proved in \cite{bm} for compact sets and the same proof works for precompact sets.

\begin{lemma}\label{lem:2}
Let $A,B$ be precompact sets with non-empty interior. Then  $\| \cdot \|_{A}$ and $\| \cdot \|_{B}$ are equivalent norms on $\cM^\infty(G)$.
\end{lemma}
\begin{proof} It is obvious that $\| \cdot \|_{A}$ defines a semi-norm, and since it has non-empty interior it is a norm.

Now, since $A$ and $B$ are precompact and have non-empty interior, each set can be covered by finitely many translates of the other. Let $N$ be the number of translates needed for both coverings. Then, it is straightforward to see that
\begin{displaymath}
\frac{1}{N}\, \| \cdot \|_A \leqslant \| \cdot \|_B \leqslant N\, \| \cdot \|_A \,.          \qedhere
\end{displaymath}
\end{proof}

\smallskip

We complete the section by looking at the completion of various spaces of translation bounded measures with respect to norm and product topologies. First, let us recall the following result.

\begin{theorem}\label{cm Banach} \cite{CRS3}
Let $K \subseteq G$ be any compact set with non-empty interior. Then, the pair $(\cM^\infty(G), \| \cdot \|_K)$ is a Banach space.
\end{theorem}

Now, Lemma~\ref{lem:2} and Theorem~\ref{cm Banach} imply the next corollary.

\begin{coro}Let $U \subseteq G$ be any open precompact set. Then, $(\cM^\infty(G), \| \cdot \|_U)$ is a Banach space.
\end{coro}

\smallskip

We next show that the spaces of translation bounded measures of spectral purity are closed in $(\cM^\infty(G), \| \, \|_U)$ and hence Banach spaces. Let us start with the following result.

\begin{lemma}\label{lem:3}
For all $\alpha \in \{ pp, ac, sc \}$ and all $\mu \in \cM^\infty(G)$ we have
\begin{displaymath}
\| \mu_{\alpha} \|_U \leqslant \| \mu \|_U \leqslant \| \mu_{\text{pp}} \|_U +  \| \mu_{\text{ac}} \|_U+ \| \mu_{\text{sc}} \|_U \,.
\end{displaymath}
\end{lemma}
\begin{proof} We follow the idea of \cite[Cor.~8.4]{NS12}. By \cite[Thm.~14.22]{HR},
we have
\begin{displaymath}
\left| \mu \right| =\left| \mu_{\text{pp}} \right|+\left| \mu_{\text{ac}} \right|+\left| \mu_{\text{sc}} \right| \,.
\end{displaymath}
The claim follows immediately from this.
\end{proof}

The following statements are immediate consequences of Lemma~\ref{lem:3}
\begin{coro}
One has
\begin{displaymath}
\cM^\infty(G)=\cM^\infty_{\text{pp}}(G)\oplus \cM^\infty_{\text{ac}}(G) \oplus \cM^\infty_{\text{sc}}(G) \,.
\end{displaymath}
\end{coro}

\begin{coro}\label{coro:1}
Let $\mu, \mu_n \in \cM^{\infty}(G)$, for all $n\in\N$. Then, $\limn \| \mu-\mu_n \|_U = 0$ if and only if
\begin{displaymath}
\limn  \| \left( \mu-\mu_n \right)_{\alpha} \|_U =0\quad \ \text{ for all } \alpha \in  \{ \text{pp}, \text{ac} , \text{sc} \} \,.
\end{displaymath}
\end{coro}

\begin{coro}\label{coro:2}
Let $\mu \in \NAP(G)$. Then $ \mu_{\text{pp}}, \mu_{\text{ac}},\mu_{\text{sc}} \in \NAP(G)$.
\end{coro}

We can now prove the following result.

\begin{prop}\label{prop:1}
The spaces $\cM_{\text{pp}}^\infty(G),\, \cM_{\text{ac}}^\infty(G)$ and $\cM_{\text{sc}}^\infty(G)$ are closed in $(\cM^\infty(G), \| \cdot \|_U)$. In particular, $(\cM_{\text{pp}}^\infty(G), \| \cdot \|_U), (\cM_{\text{ac}}^\infty(G), \| \cdot \|_U)$ and  $(\cM_{\text{sc}}^\infty(G), \| \cdot \|_U)$ are Banach spaces.
\end{prop}
\begin{proof}
Let $\alpha \in \{ \text{pp}, \text{ac}, \text{sc} \}$, and let $(\mu_n)_{n\in\N}$ be a sequence in $\cM^\infty_{\alpha}(G)$. Assuming that $\mu_n \to \mu$ in $\cM^\infty(G)$, we need to show that $\mu \in \cM^\infty_{\alpha}(G)$.

Now, if $\beta \in \{ \text{pp}, \text{ac}, \text{sc} \}$ and $\beta \neq \alpha$, we have $\left( \mu_n \right)_\beta =0$. Therefore, by Corollary~\ref{coro:1}, we  get $ \| \mu_{\beta} \|_U=0$ and hence $\mu_{\beta} =0$. As $\mu _\beta=0$ for all $\beta \neq \alpha$, $\beta \in \{ \text{pp}, \text{ac}, \text{sc} \}$, we get $\mu \in \cM^\infty_{\alpha}(G)$, as claimed.
\end{proof}

We complete this section by showing that the spaces of pure point, absolutely continuous and singular continuous measures, respectively, are not closed in the product topology. To do so, we first provide a simple lemma which simplifies some of our computations below.

\begin{lemma}\label{lem:4}
Let $\mu_n,\, \mu \in \cM^\infty(G)$, for all $n\in\N$, such that $\mu_n \xrightarrow{\pi} \mu$. If $\nu$ has compact support, we have
\begin{displaymath}
 \mu_n*\nu \xrightarrow{\pi} \mu*\nu \,.
\end{displaymath}
\end{lemma}
\begin{proof} Let $g \in \Cc(G)$. Since $\nu$ has compact support, we have $f:= g*\nu \in \Cc(G)$. As $\mu_n \to \mu$ in the product topology and $f \in \Cc(G)$, we get
\begin{displaymath}
\limn \| \mu_n *f - \mu*f \|_\infty =0 \,.
\end{displaymath}
Therefore, we have
\begin{displaymath}
\limn \| (\mu_n*\nu) *g - (\mu*\nu)*g \|_\infty =0 \,.
\end{displaymath}
As $g \in \Cc(G)$ is arbitrary, the claim follows.
\end{proof}

\smallskip

Now, we look at some examples.

\begin{example}\label{ex1}
Let $\mu_n =\frac{1}{n}\sum_{k=1}^n \delta_{\frac{k}{n}}$, for all $n\in\N$. Then, $\mu_n \xrightarrow{\pi} \lm_{[0,1]}$.
\end{example}
\begin{proof}
Let $f \in \Cc(G)$. Each such $f$ is uniformly continuous. Fix $\eps>0$. Then, there is $N\in\N$ such that
\begin{align*}
\left| (f*\lm|_{[0,1]})(x) - (f*\mu_n)(x)\right|
    &=\left|\int_{0}^1 f(x-y)\ \dd y - \frac{1}{n} \sum_{k=1}^n f\big(x-
        \frac{k}{n}\big) \right|  \\
     &=\left|\sum_{k=1}^n \int_{\frac{k-1}{n}}^{\frac{k}{n}} f(x-y)\ \dd
        y -  \sum_{k=1}^n \int_{\frac{k-1}{n}}^{\frac{k}{n}} f\big(x-
        \frac{k}{n}\big)\ \dd y \right|  \\
    &\leqslant \sum_{k=1}^n \int_{\frac{k-1}{n}}^{\frac{k}{n}}  \left|f(x-y) -
       f\big(x- \frac{k}{n}\big) \right|  \dd y \\
    &< \sum_{k=1}^n \int_{\frac{k-1}{n}}^{\frac{k}{n}} \eps\ \dd y  \, =\, \eps
\end{align*}
holds independently of $x$, for all $n\geqslant N$, due to the uniform continuity of $f$. Therefore, for all $n >N$, we have
\begin{displaymath}
\|f*\lm|_{[0,1]}- f*\mu_n\|_{\infty} \leqslant \eps \,.   \qedhere
\end{displaymath}
\end{proof}

\smallskip

\begin{example}\label{ex4}
Let $\nu$ be any singular continuous measure of compact support. Then, with $\mu_n$ as in Example \ref{ex1}, the measures $\mu_n*\nu$ are singular continuous, and the sequence $(\mu_n*\nu)_{n\in\N}$ converges in the product topology to the absolutely continuous measure $\nu*\lm_{[0,1]}$.
\end{example}
\noindent\textit{Proof.}
It is easy to see that $\mu_n*\nu$ is a finite sum of singular continuous measures, which is again singular continuous.

Next, $\nu*\lm|_{[0,1]}$ is an absolutely continuous measure because
\begin{align*}
(\nu*\lm|_{[0,1]})(\phi)
    &= \int_{\R} \int_{\R} \phi(x+y)\, 1_{[0,1]}(x)\ \dd\lm(x)\, \dd\nu(y)  \\
    &= \int_{\R} \int_{\R} \phi(x)\, 1_{[0,1]}(x-y)\ \dd\lm(x)\, \dd\nu(y)  \\
    &= \int_{\R} \int_{\R}\, 1_{[0,1]}(x-y)\ \dd\nu(y)\ \phi(x)\ \dd\lm(x)  \\
    &= \int_{\R} h(x)\, \phi(x)\ \dd\lm(x),
\end{align*}
where $h(x):= (1_{[0,1]}*\nu)(x)$.

Finally, by Lemma~\ref{lem:4}, the sequence $(\mu_n*\nu)_{n\in\N}$ converges in the product topology to $\nu*\lm_{[0,1]}$. \exend

\begin{example} Consider the following measures on $\R^2$:
\begin{displaymath}
\mu_n =\frac{1}{n}\sum_{k=1}^n \delta_{(\frac{k}{n},0)},\quad n\in\N \,.
\end{displaymath}
Then, exactly as in Example~\ref{ex1}, it can be shown that $(\mu_n)_{n\in\N}$ converges in the product topology to the singular continuous measure
\begin{displaymath}
 \lambda(f)= \int_0^1 f(x,0)\ \dd x
\end{displaymath}
for all $f \in \Cc(\R^2)$. \exend
\end{example}

\begin{example} \label{ex2}
Let $(f_\alpha)_{\alpha}$ be an approximate identity for $(\Cu(G), \| \cdot \|_\infty)$. Then, the net $(\mu_{\alpha})_{\alpha}$, with $\mu_\alpha = f_\alpha\, \theta_G$, converges in the product topology to $\delta_0$.
\end{example}
\noindent\textit{Proof.}
Since $\mu_{\alpha}*f=f_{\alpha}*f$ and $\delta_0*f=f$ for all $f\in\Cc(G)$, the claim follows from \cite[Thm. 1.2.19(b)]{Gra}.  \exend

\begin{example}\label{ex5}
Let $\nu$ be any singular continuous measure of compact support. Then, with $f_\alpha$ as in Example \ref{ex2}, the measures $\mu_\alpha*\nu$ are absolutely continuous, and $(\mu_{\alpha}*\nu)_{\alpha}$ converges in the product topology to the singular continuous measure $\nu$.
\end{example}
\noindent\textit{Proof.}
It is obvious that $\mu_\alpha*\nu$ are absolutely continuous, and by Lemma~\ref{lem:4}, we have
\begin{displaymath}
\lim_\alpha \mu_\alpha*\nu =\delta_0* \nu =\nu
\end{displaymath}
in the product topology.  \exend

\medskip

Next, we provide a slight generalisation of Example~\ref{ex2}.

\begin{lemma}\label{lemm gen approximate identity}
Let $(\mu_n)_{n\in\N}$ be a sequence of probability measures such that, for each $\eps >0$, there exists some $N\in\N$ such that, for all $n > N$, we have $\supp(\mu_n) \subseteq B_{\eps}(0)$. Then, we have
\begin{displaymath}
\limn \mu_n = \delta_0
\end{displaymath}
in the product topology.
\end{lemma}
\begin{proof}
First note that every second countable locally compact Abelian group is metrisable.
Let $f \in \Cc(G)$ be arbitrary, and fix $\eps >0$. Since $f$ is uniformly continuous, there exists some $\delta >0$ such that, for all $x,y \in G$ with $d(x,y) < \delta$, we have $\left|f(x)-f(y) \right| < \eps$.
By the condition on the support, there exists some $N\in\N$ such that, for all $n >N$, we have $\supp(\mu_n) \subseteq B_{\delta}(0)$.
Then, since each $\mu_n$ is a probability measure, for all $n >N$ and all $x \in G$, we have
\begin{align*}
\left| (\mu_n*f)(x) - (\delta_0*f)(x) \right|
    &= \left| \int_G f(x-y)\ \dd \mu_n(y) -f(x) \right|  \\
    &= \left| \int_G f(x-y)\ \dd \mu_n(y) - \int_G f(x)\ \dd \mu_n(y) \right| \\
    &\leqslant  \int_G \left|  f(x-y) - f(x) \right|\ \dd \mu_n(y) \\
    &=  \int_{B_\delta(0)}  \left|  f(x-y) - f(x) \right| \dd \mu_n(y)
\end{align*}
where we used $\supp(\mu_n) \subseteq B_{\delta}(0)$ in the last step.
Now, since $y \in B_\delta(0)$, we get $d(y,0) < \delta$ and hence, since the metric is translation invariant, $d(x-y,y) <\delta$. By the uniform continuity of $f$, this gives $\left|  f(x-y) - f(x) \right| < \eps$ and thus
\begin{displaymath}
  \left| (\mu_n*f)(x) - (\delta_0*f)(x) \right|  < \eps \
\end{displaymath}
for all $x\in G$ and $n>N$, which implies the claim.
\end{proof}

\begin{remark}
\begin{itemize}
\item[(i)] In Lemma~\ref{lemm gen approximate identity}, the condition that each $\mu_n$ is a probability measure can be weakened to $\mu_n(G)=1$ and there exists some $C >0$ such that $\left| \mu_n \right|(G) <C$, for all $n\in\N$.
\item[(ii)] Let $G$ be an arbitrary LCAG, which is not necessarily metrisable, and let $(\mu_\alpha )_\alpha$ be a net of probability measures such that, for each open set $0 \in V \subseteq G$, there exists some $\beta$ such that, for all $\alpha >  \beta$, we have $\supp(\mu_\alpha) \subseteq V $. Then, exactly as in the proof of Lemma~\ref{lemm gen approximate identity}, it can be shown that $(\mu_\alpha)_{\alpha}$ converges in the product topology to $\delta_0$.
\end{itemize}\exend
\end{remark}

Now, we provide two examples of singular continuous measures which converge to pure point measures.

\begin{example}\label{ex6}
For all $n\in\N$, let $\mu_n$ be the normalised line integral over the circle $x^2+y^2=\frac{1}{n^2}$ in $\R^2$, that is
\begin{displaymath}
\mu_n(f)=\frac{1}{2 \pi} \int_{0}^{2 \pi} f\Big( \frac{\cos(t)}{n}, \frac{\sin(t)}{n} \Big)\ \dd t \,.
\end{displaymath}
Then, every $\mu_n$ is a singular continuous measure, and by Lemma~\ref{lemm gen approximate identity}, the sequence $(\mu_n)_{n\in\N}$ converges in the product topology to $\delta_0$.      \exend
\end{example}

\begin{example}\label{ex7}
Let $\mu$ be any singular continuous probability measure on $\R$ supported inside $[0,1]$.
Define $\mu_n$, for all $n\in\N$, by
\begin{displaymath}
\mu_n(f)=\int_{\R} f(nx)\ \dd \mu(x) \,.
\end{displaymath}
Then, each $\mu_n$ is a singular continuous probability measures supported inside $[0, \frac{1}{n}]$, and therefore, by Lemma~\ref{lemm gen approximate identity}, $(\mu_n)_{n\in\N}$ converges in the product topology to $\delta_0$.\exend
\end{example}

\section{Strong versus norm almost periodicity}\label{SAP vs NAP}

The purpose of this section is to show that norm almost periodicity is a uniform version of strong almost periodicity.

Recall that, for a measure $\mu$, we can define
\begin{displaymath}
P_{\varepsilon}^U(\mu):= \{ t \in G\ |\ \| T_t \mu -\mu \|_U \leqslant \varepsilon \} \,.
\end{displaymath}
Similarly, for a $f \in \Cu(G)$, we can define
\begin{displaymath}
P_{\varepsilon}(f):= \{ t \in G\ |\ \| T_t f -f \|_\infty \leqslant \varepsilon \} \,.
\end{displaymath}

\begin{remark}
\begin{itemize}
\item [(i)] Sometimes the set of $\eps$-almost periods is defined with strict inequality. It is easy to see that the notion of almost periodicity is independent of the choice of $\leqslant$ or $<$.
\item [(ii)] Usually, the norm $\| \cdot \|_K$ and norm almost periodicity are defined using compact sets $K$. Working with open precompact sets makes our computations below much simpler.
\end{itemize}\exend
\end{remark}

Let us start with the following lemma.

\begin{lemma}\label{C2} Let $U \subseteq G$ be an open precompact set, and let $\cF \subseteq \cF_U$ be any set which is dense in $(\cF_U, \| \cdot \|_U)$. Then, we have
\begin{displaymath}
P_{\varepsilon}^U(\mu)= \{ t \in G\ |\ \| T_t (\mu*g^{\dagger}) -\mu*g^{\dagger} \|_{\infty} \leqslant \varepsilon\ \text{ for all }g \in \cF \} = \bigcap_{g\, \in\, \cF} P_{\varepsilon}(\mu*g^{\dagger}) \,.
\end{displaymath}
In particular
\begin{displaymath}
P_{\varepsilon}^U(\mu)= \{ t \in G\ |\ \| T_t (\mu*g^{\dagger}) -\mu*g^{\dagger} \|_{\infty} \leqslant \varepsilon\ \text{ for all }g \in \cF_U \} = \bigcap_{g\, \in\, \cF_U} P_{\varepsilon}(\mu*g^{\dagger}) \,.
\end{displaymath}
\end{lemma}
\begin{proof}
Let
\begin{align*}
B_{\varepsilon}&:= \{ t \in G\ |\ \| T_t (\mu*g^{\dagger}) -\mu*g^{\dagger} \|_{\infty} \leqslant \varepsilon\ \text{ for all }g \in \cF \} \,, \\
C_{\varepsilon}&: = \bigcap_{g\, \in\, \cF} P_{\varepsilon}(\mu*g^{\dagger}) \,.
\end{align*}
The equality $P^U_\varepsilon(\mu)=B_{\varepsilon}$ is an immediate consequence of Corollary~\ref{C5} because
\begin{align*}
\|T_t\mu-\mu\|_{U}
     &= \sup_{(x,g)\in G\times \cF} \left| (T_t\mu-\mu)(T_xg)\right|  = \sup_{(x,g)\in G\times \cF} \left| \mu(T_{x-t}g) - \mu(T_xg)\right|  \\
     &= \sup_{(x,g)\in G\times \cF} \left| \big(T_t(\mu*g^{\dagger})\big)(x) - (\mu*g^{\dagger})(x)\right|  \\
     &= \sup_{g\in\cF} \sup_{x \in G}  \left| \big(T_t(\mu*g^{\dagger})\big)(x) - (\mu*g^{\dagger})(x)\right|  = \sup_{g\in\cF} \| T_t(\mu*g^{\dagger}) - \mu*g^{\dagger} \|_{\infty} \,.
\end{align*}
Therefore, we have
\begin{align*}
t\in P_{\varepsilon}^{U}(\mu)
    &\iff \|T_t\mu-\mu\|_U \leqslant \varepsilon \\
    &\iff \sup_{g\in\cF} \| T_t(\mu*g^{\dagger}) - \mu*g^{\dagger} \|_{\infty} \leqslant \varepsilon \\
    &\iff \|T_t(\mu*g^{\dagger}) - \mu*g^{\dagger} \|_{\infty}\leqslant \varepsilon \ \ \text{ for all }g\in\cF_{U}  \\
    &\iff t \in B_{\varepsilon}.
\end{align*}
The equality $B_{\varepsilon}=C_{\varepsilon}$ follows immediately from the definition of $P_\eps(\mu*g^\dagger)$.
\end{proof}

\begin{remark}
If $U$ is symmetric, i.e. $-U=U$, we can replace $g^{\dagger}$ by $g$.
Since we are interested in norm almost periodicity, which by Lemma~\ref{lem:2} does not depend on the choice of $U$, one can assume without loss of generality that $U=-U$, to make the computations below slightly simpler. Since this assumption doesn't simplify the formulas too much, and in future applications one may need to work without this extra assumption, we don't assume below that $U$ is symmetric.\exend
\end{remark}

\medskip

Next, we show that to check that a measure is strongly almost periodic, it suffices to use $\cF_U$ as the set of test functions.

\begin{prop}  \label{prop:char_sap}
Let $\mu \in \cM^\infty(G)$. Then, $\mu \in \SAP(G)$ if and only if $\mu*g \in \operatorname{SAP}(G)$ for all $g \in \cF_U$.
\end{prop}
\begin{proof}
By definition, the given property is necessary for $\mu$ to be a strongly almost periodic measure. It is also sufficient. If $f \in \Cc(G)$ it is easy to show that there exist elements $c_1,\ldots, c_m \in \CC$, $t_1,\ldots, t_m \in G$ and $g_1,\ldots, g_m \in \cF_U$ such that
\begin{equation}\label{EQ1}
  f= \sum_{j=1}^m c_j T_{t_j} g_j \,.
\end{equation}
Indeed, since $\supp(f)$ is finite, we can find a finite open cover $\supp(f) \subseteq \bigcup_{j=1}^m (-t_j+U)$. By a standard partition of unity argument, we can find $h_1,.., h_m \in \Cc(G)$ such that $\sum_{j=1}^m h_j(x)=1$ for all $x \in \supp(f)$. Then $c_j = \| fh_j \|_\infty$ and $g_j =\frac{1}{c_j} f\cdot (T_{-t_j}h_j)$ for all $j$ such that $c_j \neq 0$ gives Eq. \eqref{EQ1}.
The claim is now obvious.
\end{proof}

We now introduce the concept of equi-Bohr almost periodicity.

\begin{definition}
Let $\mathcal{G} \subseteq \operatorname{SAP}(G)$ be any family of functions. We say that $\mathcal{G}$ is \textbf{equi-Bohr almost periodic} if, for each $\eps >0$, the set
\begin{displaymath}
P_\eps(\mathcal{G})= \{ t \in G\ |\ \|T_tf-f \|_\infty \leqslant \eps\  \text{ for all } f \in \mathcal{G} \}
\end{displaymath}
is relatively dense.
\end{definition}

\begin{remark} It is easy to see that
$
P_\eps(\mathcal{G})=\bigcap_{f\, \in\, \mathcal{G}} P_\eps(f).
$\exend
\end{remark}

\smallskip

We can now prove the main result in this section.

\begin{theorem}  \label{thm:char_nap}
Let $\mu \in \cM^\infty(G)$, let $U\subseteq G$ be an open precompact set, and let $\mathcal F \subseteq \cF_U$ be dense in $(\cF_U, \| \cdot \|_\infty)$. Then, $\mu$ is norm almost periodic if and only if $\mathcal{G}_{\cF}:=\{ \mu *g\ |\ g \in \cF \}$ is equi-Bohr almost periodic.

In particular, $\mu$ is norm almost periodic if and only if the family $\mathcal{G}:=\{ \mu *g\ |\ g \in \cF_U \}$ is equi-Bohr almost periodic.
\end{theorem}
\begin{proof}
This is an immediate consequence of Lemma~\ref{C2}. Indeed, we have
\begin{displaymath}
P_\eps(\mathcal{G}_{\cF})=\bigcap_{g \in \cF} P_\eps(\mu*g) = P_{\varepsilon}^{-U}(\mu) \,.   \qedhere
\end{displaymath}
\end{proof}

\smallskip

\begin{remark}
By combining Proposition~\ref{prop:char_sap} and Theorem~\ref{thm:char_nap}, the following are true for a measure $\mu\in\mathcal{M}^{\infty}(G)$.
\begin{enumerate}
\item[(i)] The measure $\mu$ is strongly almost periodic iff, for all $\varepsilon>0$, the set $P_{\varepsilon}(\mu*g)$ is relatively dense in $G$ for all $g\in \cF_U$.
\item[(ii)] The measure $\mu$ is norm almost periodic iff, for all $\varepsilon>0$, the set $\bigcap_{g\,\in\,\cF_U} P_{\varepsilon}(\mu*g)$ is relatively dense in $G$.
\end{enumerate}\exend
\end{remark}

We next use the results from this section to give simpler proofs for \cite[Prop.~6.2]{NS12} and \cite[Prop.~5.6]{LSS}.

\begin{prop} \cite[Prop.~6.2]{NS12}\label{p2} Let $\mu$ be a norm almost periodic measure and $\nu$ a finite measure. Then, $\mu*\nu$ is norm almost periodic.
\end{prop}
\begin{proof}
Let $\eps>0$. Since $\mu$ is norm almost periodic, there is a relatively dense set $S$ such that $\| f*\mu-T_t (f*\mu) \|_\infty<\frac{\eps}{| \nu | (G)+1}$ for all $t\in S$ and $f \in \cF_U$.
Hence, for all $f \in \cF_U$ we have
\begin{align*}
\| f*(\mu*\nu)-T_t f*(\mu *\nu)\|_\infty
   & =\| (f*\mu-T_t (f*\mu)) *\nu\|_\infty \\
   &=\sup_{x \in G}\, \left| \int_G (f*\mu-T_t( f*\mu))(x-y)\ \dd \nu(y)
      \right| \\
   &\leqslant \sup_{x \in G}\, \int_G  \big|(f*\mu-T_t( f*\mu))(x-y) \big|\ \dd
      |\nu|(y) \\
   &\leqslant \| f*\mu-T_t (f*\mu) \|_\infty \, | \nu | (G) < \eps\,.
\end{align*}
Consequently, $\mu*\nu$ is norm almost periodic.
\end{proof}

Next, we give an alternate proof for the following result.

\begin{prop}\cite{LSS}
Let $\mu,\nu\in \mc{M}^{\infty}(G)$, and let $\mc{A}=(A_n)_{n\in\N}$ be a van Hove sequence such that $\mu\circledast_{\mc{A}}\nu$ exists\footnote{Recall that $\mu\circledast_{\mc{A}}\nu$ is defined as the vague limit of $\{ \frac{1}{|A_n|}\mu|_{A_n}*\nu_{A_n}\}$, if the limit exists. Given a van Hove sequence, the limit always exists along a subsequence \cite{LSS}.}. If $\mu$ is norm almost periodic, then $\mu\circledast_{\mc{A}}\nu$ is norm almost periodic.
\end{prop}
\begin{proof}
By \cite[Lem. 1.1]{Martin2}, there is a constant $c>0$ such that, for all $g \in \cF_{U}$, we have
\begin{align*}
\|\left(\mu\circledast_{\mc{A}}\nu \right.&- \left. T_t(\mu\circledast_{\mc{A}}\nu)\right)*g\|_\infty \\
    &\leqslant \sup_{x \in G} \lims \frac{1}{|A_n|} \int_{A_n} \left| \big(\mu*g
         - T_t (\mu*g)\big) (x-t)\right|\  \dd |\nu|(t) \\
    &\leqslant  \| \mu* g  - T_t (\mu*g) \|_\infty\, \sup\Big\{\frac{ | \nu
       |(A_n)}{|A_n|} \ \Big|\ n\in\N \Big\}   \\
    &\leqslant c\, \|\mu-T_t\mu\|_{\infty} \,.
\end{align*}
Now, this and Corollary~\ref{C1} imply
\begin{align*}
\| \mu\circledast_{\mc{A}}\nu - T_t(\mu\circledast_{\mc{A}}\nu)\|_U
    &= \sup_{g\in\cF_U} \sup_{s\in G} \left| \big(\mu\circledast_{\mc{A}}\nu
       -T_t (\mu\circledast_{\mc{A}}\nu)\big)(T_sg) \right|  \\
    &= \sup_{g\in\cF_U} \sup_{s\in G} \left| \big(\big(\mu\circledast_{\mc{A}}
        \nu -T_t (\mu\circledast_{\mc{A}}\nu)\big)*g^{\dagger}\big)(s) \right|
        \\
    &\leqslant c\, \|\mu-T_t\mu\|_{\infty} \,,
\end{align*}
which finishes the proof.
\end{proof}

\smallskip
We next look at the completeness of $\NAP(G)$.

\begin{prop}
$\NAP(G)$ is closed in $(\cM^\infty(G), \| \cdot \|_U)$. In particular, $(\NAP(G), \| \cdot \|_U)$ is complete.
\end{prop}
\begin{proof}
Let $(\mu_n)_{n\in\N}$ be a sequence in $\NAP(G)$, and let $\mu \in \cM^\infty(G)$ be such that $\mu_n \to \mu$ in $\| \cdot \|_U$. Let $\eps>0$. Then, there exists some $n\in\N$ such that $\| \mu -\mu_n \|_U <\frac{\eps}{3}$.
Since $\mu_n \in \NAP(G)$ the set $P:= \{ t \in G \ |\ \| \mu_n - \mu \|_U < \frac{\eps}{3} \}$ is relatively dense. Moreover, for all $t \in P$ we have
\begin{align*}
\| T_t \mu - \mu \|_U
    &= \| T_t \mu - T_t \mu_n \|_U +\| T_t \mu_n - \mu_n \|_U +\|\mu_n - \mu \|_U   \\
    &<\frac{\eps}{3}+\frac{\eps}{3}+\frac{\eps}{3} \, = \, \eps \,.   \qedhere
\end{align*}
\end{proof}

\smallskip

However, note that $\NAP(G)$ is not a Banach space because it is not closed under addition.
Since the intersection of two closed subsets of a topological space is also closed, we get the following consequence.

\begin{coro}\label{C3} For each $\alpha \in \{ pp, ac, sc \}$ the set
\begin{displaymath}
\NAP_{\alpha}(G):= \NAP(G) \cap \cM_{\alpha}^\infty(G)
\end{displaymath}
is closed in $(\cM^\infty(G), \| \cdot \|_U)$.
\end{coro}

\section{Strong almost periodicity and Lebesgue decomposition}\label{sap leb}

Recall that, by Corollary~\ref{coro:2}, given a measure $\mu \in \NAP(G)$ we have $\mu_{\text{pp}}, \mu_{\text{ac}}, \mu_{\text{sc}} \in \NAP(G)$. In this section, we show that the same does not hold for strong almost periodic measures. First, we will prove the following lemma, which will be our main tool for constructing examples.

\begin{lemma} Let $\mu_n, \mu$ be measures on $\R$ supported inside $[0,1]$, for all $n\in\N$, such that $\mu_n \xrightarrow{\pi} \mu$. Define
\begin{displaymath}
\omega:= \mu+\sum_{j=1}^\infty \left(\delta_{2^j+(2^{j+1}\Z)}\right) * \mu_j \,.
\end{displaymath}
Then, $\omega \in \SAP(\R)$.
\end{lemma}
\begin{proof}

Let $f \in \Cc(\R)$, and let $\eps >0$. Let $M \in \N$ be such that $\supp(f) \subseteq [-M,M]$.
Since the sequence $(\mu_j)_{j\in\N}$ converges in the product topology to $\mu$, there exists some $N\in\N$ such that, for all  $j >N$, we have
\begin{equation}\label{EQ44}
\| \mu_j*f- \mu*f \|_\infty < \frac{\eps}{4M+2} \,.
\end{equation}
In particular, for all $ j,m >N$, we have
\begin{equation}\label{EQ55}
\| \mu_j*f- \mu_m*f \|_\infty < \frac{\eps}{2M+1} \,.
\end{equation}
Next, we show that every element of $2^{N+1}\Z$ is an $\eps$-almost period for $\omega*f$.

Define
\begin{displaymath}
\omega_N:= \sum_{j=1}^N \delta_{2^j+(2^{j+1}\Z)} * \mu_j \,.
\end{displaymath}
Then, $\omega_N$ is $2^{N+1}\Z$ periodic. Therefore, to show that $2^{N+1}\Z$ are $\eps$-almost periods for $\omega*f$, it suffices to show that $2^{N+1}\Z$ are $\eps$-almost periods for $(\omega-\omega_N)*f$. Now,
\begin{displaymath}
\omega- \omega_N= \mu+\sum_{j=N+1}^\infty \delta_{2^j+(2^{j+1}\Z)} * \mu_j \,.
\end{displaymath}
Next, if we define
\begin{displaymath}
\omega_n :=
\begin{cases}
\mu_{v_2(n)}, & \text{if } n\in\Z\setminus \{0\},  \\
\mu, & \text{if } n=0,
\end{cases}
\end{displaymath}
where $v_2(n)$ is the $2$-adic valuation of $n$, we can write
\begin{displaymath}
\omega- \omega_N = \sum_{n \in \Z} \delta_{2^{N+1}\cdot n} * \omega_n \,.
\end{displaymath}
Then, for all $k \in \Z$, we have
\begin{align*}
(\omega- \omega_N)*f - T_{2^{N+1}\cdot k} \left( (\omega- \omega_N)*f \right)
    & =\sum_{n \in \Z} \delta_{2^{N+1} \cdot n} * \omega_n *f- T_{2^{N+1}\cdot k}
       \left( \sum_{n \in \Z} \delta_{2^{N+1} \cdot n}* \omega_n *f\right) \\
    &=\sum_{n \in \Z} \delta_{2^{N+1} \cdot n} *\left(\omega_n *f-
       \omega_{n-k}*f \right)
\end{align*}
Now, by Eqs.~\eqref{EQ44} and \eqref{EQ55}, we have
\begin{displaymath}
\|  \omega_n *f- \omega_{n-k}*f  \|_\infty < \frac{\eps}{2M+1} \,.
\end{displaymath}
Using the fact that $\supp(f) \subseteq [-M,M]$, we immediately get
\begin{displaymath}
\| (\omega- \omega_N)*f - T_{2^N\cdot k} \left( (\omega- \omega_N)*f \right) \|_\infty <\eps \,.
\end{displaymath}
This completes the proof.
\end{proof}

\smallskip

Now, Examples~\ref{ex1},~\ref{ex2},~\ref{ex5},~\ref{ex6} and Example~\ref{ex7} yield the following example of strongly almost periodic measures for which the components of the Lebesgue decomposition are not strongly almost periodic.

\begin{example}\label{ex3} Let
\begin{displaymath}
\mu:= \lm_{[0,1]}+\sum_{j=1}^\infty \left(\delta_{2^j+(2^{j+1}\Z)} * \bigl( \frac{1}{j} \sum_{k=0}^{j-1} \delta_{\frac{k}{j}} \bigr)\right)
\end{displaymath}
Then, $\mu \in \SAP(\R)$. \exend
\end{example}

\smallskip

\begin{example}\label{ex9}
Let $\nu$ be a singular continuous measure supported inside $[0,1]$. Define
\begin{align*}
\mu(f):&= \int_0^1 f(2x)\ \dd (\nu*\lm_{[0,1]})(x)\,, \\
\mu_n(f):&= \int_0^1 f(2x)\ \dd \big(\nu*(\frac{1}{n}\sum_{k=1}^n \delta_{\frac{k}{n}})\big)(x)\,.
\end{align*}
Then, $(\mu_n)_{n\in\N}$ is a sequence of singular continuous measures supported inside $[0,1]$ which, by Example~\ref{ex2}, converge in the product topology to the absolutely continuous measure $\mu$. As $\supp(\mu)\subseteq [0,1]$, it follows that
\begin{displaymath}
\omega:= \mu+\sum_{j=1}^\infty \delta_{2^j+(2^{j+1}\Z)} * \mu_j \in \SAP(\R) \,.
\end{displaymath}\exend
\end{example}

\begin{example}\label{ex8}
For all $n\in\N$, let
\begin{displaymath}
f_n(x):= \max \{ n-n^2\,|x| , 0 \}\quad \text{ and } \quad \mu_n:=f_n\,\lm\,.
\end{displaymath}
Then, it is trivial to see that $(f_n)_{n\in\N}$ is an approximate identity for the convolution on $\R$. Therefore, by Example~\ref{ex2}, we have
\begin{displaymath}
\mu= \delta_0+ \sum_{j=1}^\infty \delta_{2^j+(2^{j+1}\Z)} * \mu_j \in \SAP(\R) \,.
\end{displaymath}
This is a measure with non trivial pure point and absolutely continuous components, and trivial singular continuous component. Since $\mu_{\text{pp}}=\delta_0$, neither $\mu_{\text{pp}}$ note $\mu_{\text{ac}}$ is strongly almost periodic.\exend
\end{example}

\begin{example}
Let $\nu$ be a singular continuous measure supported inside $[0,1]$, and let $f_n$ be as in Example \ref{ex8}. Define
\begin{align*}
\mu(f):&= \int_0^1 f(2x)\ \dd \nu (x)\,, \\
\mu_n(f):&= \int_0^1 f(2x)\ \dd (\nu*f_n)(x)\,.
\end{align*}
Then, $(\mu_n)_{n\in\N}$ is a sequence of absolutely continuous measures supported inside $[0,1]$ which, by Example~\ref{ex5}, converges in the product topology to the singular continuous measure $\nu$. As $\supp(\nu)\subseteq [0,\frac{1}{2}]\subseteq [0,1]$, it follows that
\begin{displaymath}
\omega:= \mu+\sum_{j=1}^\infty \delta_{2^j+(2^{j+1}\Z)} * \mu_j \in \SAP(\R) \,.
\end{displaymath}
\end{example}\exend

\begin{example}
Let $\mu$ be any singular continuous probability measure on $\R$ supported inside $[0,1]$.
Define $\mu_n$ by
\begin{displaymath}
\mu_n(f)=\int_{\R} f(nx)\ \dd \mu(x) \,,
\end{displaymath}
for all $n\in\N$. Then, $\mu_n$ is a singular continuous probability measure supported inside $[0, \frac{1}{n}]$, and by Example~\ref{ex7}
\begin{displaymath}
\omega:= \delta_0 +\sum_{j=1}^\infty \delta_{2^j+(2^{j+1}\Z)} * \mu_j \in \SAP(\R) \,.
\end{displaymath}
\end{example}\exend

\begin{example} Let $\mu$ be the measure from Example~\ref{ex3} and $\omega$ be the measure from Example~\ref{ex9}. Define
\begin{displaymath}
\nu:= \mu+\omega \,.
\end{displaymath}
Then $\nu \in \SAP(G)$ but it follows from Examples~\ref{ex3},~\ref{ex9} that neither $\nu_{\text{pp}}$ nor $\nu_{\text{sc}}$ are strongly almost periodic. Moreover, $\nu_{\text{ac}}$ has compact support, and hence it is not strongly almost periodic either.\exend
\end{example}

\section{Norm almost periodic measures of spectral purity}\label{nap sp}

Here, we briefly look at each of the sets $\NAP_{\text{pp}}(G)$, $\NAP_{\text{ac}}(G)$ and $\NAP_{\text{sc}}(G)$.

\subsection{On absolutely continuous norm almost periodic measures}

First, we give a characterisation of norm almost periodicity for absolutely continuous measures in terms of $L^1$-Stepanov almost periodicity.

Let us first recall that a function $f \in L^1_{\text{loc}}(\R)$ is called \textbf{Stepanov almost periodic} if, for each $\eps >0$, the set
\begin{displaymath}
\Big\{ t\in \R\ \big|\ \sup_{x \in \R} \int_{x}^{x+1} \left| f(s)- f(s-t) \right|\ \dd s  \leqslant \eps \Big\}
\end{displaymath}
is relatively dense. It is well known that working over intervals of arbitrary length does not change the class of Stepanov almost periodic functions \cite{Che}.

Let us first extend this definition to arbitrary locally compact Abelian groups.

\begin{definition}
Let $G$ be a LCAG, and let $U \subseteq G$ be any non-empty precompact open set. A function $f \in L^1_{\text{loc}}(G)$ is called \textbf{$L^1$-Stepanov almost periodic} (with respect to $U$) if, for each $\eps >0$, the set
\begin{displaymath}
\Big\{ t\in G\ \big|\ \sup_{x \in G} \frac{1}{|U|} \int_{x+U} \left| f(s)- f(s-t) \right|\ \dd s  \leqslant \eps \Big\}
\end{displaymath}
is relatively dense.
\end{definition}

\begin{remark}
Each non-empty precompact open set $U$  defines a norm $\| \cdot \|_U$ on the space $BL^1_{\text{loc}}(G):= L^1_{\text{loc}}(G) \cap \cM^\infty(G)$ via
\begin{displaymath}
\| f\|_U :=\sup_{x \in G} \frac{1}{|U|}\int_{x+U} \left| f (s) \right|\ \dd s \,.
\end{displaymath}
An immediate computation shows that $BL^1_{\text{loc}}(G)= \{ f \in  L^1_{\text{loc}}(G) \ |\ \| f \|_U <\infty \}$. Moreover, any $L^1$-Stepanov almost periodic function belongs to $BL_{\text{loc}}^1(G)$, see \cite{Spi}.

It is easy to see that different precompact open sets define equivalent norms, and that a function $f \in  BL^1_{\text{loc}}(G)$ is $L^1$-Stepanov almost periodic if and only if, for each $\eps >0$, the set
\begin{displaymath}
\{ t \in G\ |\ \| f- T_tf \|_U \leqslant \eps \}
\end{displaymath}
is relatively dense.

Also, we will see below that the norm $\| f\|_U$ we defined here is just the measure norm $\| f\, \theta_G \|_U$.

For more details on Stepanov almost periodic functions on LCAG see \cite{Spi}.\exend
\end{remark}

\begin{lemma}\label{L3} \cite[Sec. 13.16.3]{Die}
Let $f \in L^1_{\text{\text{loc}}}(G)$ be arbitrary. Then, one has
\begin{displaymath}
\left| f\, \theta_G \right| = \left| f \right|\, \theta_G \,.
\end{displaymath}
In particular, we obtain
\begin{displaymath}
\| f\, \theta_G \|_U =  \sup_{x \in G} \int_{x+U} \left| f (s) \right|\ \dd s \,.
\end{displaymath}
\end{lemma}

The following is an immediate consequence of Proposition~\ref{prop:1}.

\begin{lemma}
The mapping $f \mapsto f \theta_G$ is an homomorphism between $(BL^1_{\text{loc}}(G), \| \cdot \|_U)$ and $(\cM^\infty_{\text{ac}}(G), \| \cdot \|_U)$. Moreover, one has
\[
\|f\, \theta_G\|_U = |U|\, \|f\|_U \,.
\]
In particular, $(BL^1_{\text{loc}}(G), \| \cdot \|_U)$ is a Banach space.
\end{lemma}

As an immediate consequence, we get the following result.

\begin{theorem}\label{T1}
An absolutely continuous translation bounded measure $\mu=f\,\theta_G$ is norm almost periodic if and only if its density function $f \in L^1_{\text{loc}}(G)$ is $L^1$-Stepanov almost periodic.

The mapping $f \mapsto f \theta_G$ is an isomorphism between the Banach spaces $(\mathcal{S}, \| \cdot \|_U)$ and $(\NAP_{\text{ac}}(G), \| \cdot \|_U)$, where
\begin{displaymath}
\mathcal{S}:= \{ f \in L^1_{\text{loc}}(G) \ |\ f \mbox{ is } L^1 \mbox{-Stepanov almost periodic} \} \,.
\end{displaymath}
\end{theorem}
\begin{proof} First, let us note that $\mathcal{S}$ is a vector space by \cite{Spi}.

It is easy to see that the above mapping is linear and onto, and therefore $\NAP_{\text{ac}}(G)$ is a vector space, which is complete by Corollary~\ref{C3}. The rest of the claims are now obvious.
\end{proof}

Finally, for measures with uniformly continuous and bounded Radon--Nikodym density, we get the following simple characterisation.

\begin{prop} Let $f \in \Cu(G)$, and let $\mu:= f\, \theta_G$. Then, the following statements are equivalent:
\begin{itemize}
  \item [(i)] $\mu$ is norm almost periodic,
  \item [(ii)] $\mu$ is strongly almost periodic,
  \item [(iii)] $f$ is Bohr almost periodic,
  \item [(iv)] $f$ is Stepanov almost periodic.
\end{itemize}
\end{prop}
\begin{proof} (i)$\iff$(iv): This is Theorem~\ref{T1}.

\smallskip

\noindent (ii)$\iff$(iii): This follows from \cite[Prop.~4.10.5 (i)]{MoSt}.

\smallskip

\noindent (i)$\iff$(iii): This follows from \cite[Prop.~5.4.6]{NS11}.
\end{proof}

\subsection{Pure point norm almost periodic measures}

The pure point norm almost periodic measures are well understood due to the following characterisation.

\begin{theorem}\label{t1}\cite{NS11,NS12}
\begin{enumerate}
 \item[(i)]Let $\mu$ be a pure point  norm almost periodic measure. Then, there exists a CPS $(G, H, \cL)$ and a continuous function $h \in \Cz(H)$ such that
\begin{displaymath}
\mu = \sum_{(x,x^\star)\, \in\, \cL} h(x^\star)\, \delta_x =: \omega_h \,.
\end{displaymath}
\item[(ii)] Let $(G, \R^d, \cL)$ be a CPS and $h \in \mathcal{S}(\R^d)$. Then, $\omega_h$ is a norm almost periodic measure.
\end{enumerate}
\end{theorem}

This allows us to construct many examples of such measures.

\subsection{Singular continuous norm almost periodic measures}

Unfortunately, we don't have a good understanding of norm almost periodic singular continuous measures.

It is easy to construct examples of such measures. Indeed, pick any pure point norm almost periodic measure $\omega$, which can be constructed by the method of Theorem~\ref{t1}. Let $\nu$ be any finite singular continuous measure. Then $\omega*\nu$ is a singular continuous measure which is norm almost periodic by Proposition~\ref{p2}.

If $\omega$ is positive and has dense support, which can easily be assured, and $\nu$ is positive, then $\omega*\nu$ has dense support.

One another hand, picking $\omega=\delta_\Z$ and $\nu$ a singular continuous measure with Cantor set support, then $\omega*\nu$ does not have dense support.

\bigskip

Recall that if the sets of norm almost periods of $\mu$ are locally finite, for $\eps$ small enough, then they are model sets in the same CPS. While this seems to be the case for many norm almost periodic singular continuous measures, it is not always true. Indeed, $\delta_\Z \otimes \lambda$ is norm almost periodic and singular continuous, but the sets of almost periods contain $\Z \times \R$.

\section{Diffraction of measures with Meyer set support}

In this section, we look at the consequences of the previous sections for the diffraction of measures with Meyer set support. For an overview of cut and project schemes and Meyer sets, and their properties, we recommend the monographs \cite{TAO,TAO2} as well as \cite{LR,Meyer,MOO,CR,CRS,NS1,NS5,NS11,NS12}.

Let us start by recalling the following result.

\begin{theorem} \cite{NS12} Let $\mu$ be any Fourier transformable measure supported inside a Meyer set. Then, each of $(\widehat{\mu})_{\text{pp}}, (\widehat{\mu})_{\text{ac}}, (\widehat{\mu})_{\text{sc}}$ is a norm almost periodic measure.
\end{theorem}

As a consequence, we can state the next corollary.

\begin{coro}  Let $\mu$ be any Fourier transformable measure supported inside a Meyer set.
\begin{enumerate}
\item[(i)] There exists some CPS $(\widehat{G}, H, \cL)$ and some $h \in \Cz(H)$ such that
\begin{displaymath}
(\widehat{\mu})_{pp} =\omega_h \,.
\end{displaymath}
\item[(ii)] There exists an $L^1$-Stepanov almost periodic function $f$ such that
\begin{displaymath}
 (\widehat{\mu})_{ac}=f\, \theta_{\widehat{G}} \,.
\end{displaymath}
\end{enumerate}
\end{coro}

\begin{remark} It follows from \cite{NS12} that there exists a CPS $(\widehat{G}, H, \cL)$ and some function $h \in \Cz(H)$ such that
\begin{displaymath}
  f  = \omega_h *f_1  \quad \text{ and }\quad
  (\widehat{\mu})_{\text{sc}}  = \omega_h* \nu
\end{displaymath}
where $f$ is the Radon--Nikodym density of the absolutely continuous part $(\widehat{\mu})_{\text{ac}}$, $f_1 \in L^1(\widehat{G})$ and $\nu$ is a finite singular continuous measure.\exend
\end{remark}

\begin{remark} Each of the examples of compatible random substitutions in one dimension covered in \cite{BSS} is a Meyer set with mixed pure point and absolutely continuous spectrum.

It follows from the general theory that there exists some CPS $(\widehat{\R}, H, \cL)$, some $h \in \Cz(H)$ and an $L^1$-Stepanov almost periodic function $f$ such that
\begin{displaymath}
\widehat{\gamma} =\underbrace{\omega_h}_{\mbox{pp}} + \underbrace{f\, \lm}_{\mbox{ac}} \,.
\end{displaymath}
Explicit formulas for both parts are provided in \cite{BSS}.\exend
\end{remark}

 \subsection*{Acknowledgments}  The work was supported by NSERC with grant 03762-2014 and by DFG, and the authors are grateful for the support.

\end{document}